\documentclass[11pt]{amsart}
\usepackage{amsmath, amsfonts, amsthm, amstext, amssymb, xspace}
\usepackage{eucal}
\usepackage{amssymb,latexsym}
\usepackage{xcolor}
\usepackage[shortlabels]{enumitem}
\pdfoutput=1 
\usepackage{microtype}
\usepackage{longtable,booktabs}
\usepackage{tikz-cd}
\usepackage{hyperref}
\usepackage[capitalise, nameinlink]{cleveref}
\hypersetup{
    colorlinks,
    linkcolor={red!50!black},
    citecolor={blue!50!black},
    urlcolor={blue!80!black}
}

\theoremstyle{definition}
\newtheorem{thm}{Theorem}[section]

\newtheorem{lemma}[thm]{Lemma}
\newtheorem{corollary}[thm]{Corollary}

\newtheorem{remark}[thm]{Remark}

\newtheorem{notation}[thm]{Notation}
\newtheorem{technique}{Technique}

\newtheorem{thmx}{Theorem}

\def\c{\mathbb{C}}

\def\f{\mathbb{F}}

\def\h{\mathbb{H}}

\def\n{\mathbb{N}}

\newcommand{\kappabar}{\bar{\kappa}}
\newcommand{\xra}{\xrightarrow}

\def\ker{\operatorname{ker}}
\def\coker{\operatorname{coker}}

\def\fib{\operatorname{fib}}

\def\tmf{\mathrm{tmf}}

\renewcommand{\h}{\operatorname{h}}

\newcommand{\mr}[1]{\mathrm{#1}}

\author{Irina Bobkova}\address{Texas A\&M University}\email{ibobkova@tamu.edu}
\author{J.D. Quigley}\address{University of Virginia}\email{mbp6pj@virginia.edu}

\title{New simple $\eta$-torsion families of elements in the stable stems}

\begin{document}

\begin{abstract}
We produce five $192$-periodic infinite families of simple $\eta$-torsion elements in the stable homotopy groups of spheres with trivial image under the $\tmf$-Hurewicz homomorphism. We also establish that several other 192-periodic families in the stable stems, which are in the $\tmf$-Hurewicz image, consist of simple $\eta$-torsion elements. 
\end{abstract}

\maketitle

\section{Introduction}

Computing the stable homotopy groups of spheres is a longstanding problem in algebraic topology, with applications in algebra, geometry, and topology. There has been significant progress in understanding the stable homotopy groups of spheres in recent years, such as the \emph{tour de force} ``low-dimensional" computations of Isaksen, Lin, Wang, and Xu \cite{Isa19, IWX20b, IWX23, LWX24} using motivic homotopy theory, the breakthrough disproof of the telescope conjecture on large-scale patterns in the stable stems by Burklund--Hahn--Levy--Schlank \cite{BHLS23}, and the analysis of geometrically significant families of elements like the Kervaire invariant one elements studied by Hill--Hopkins--Ravenel \cite{HHR16} using equivariant homotopy theory. We refer the reader to \cite{BBQ24, IWX20b, IWX23} for a more thorough discussion of the history of this problem. 

The stable homotopy groups of spheres are finite above dimension zero, so it is possible to study them one prime at a time. For the rest of this paper, we work in the $2$-local setting. There is a long tradition of producing infinite families of elements in the stable homotopy groups of spheres using Hurewicz homomorphisms to generalized homology groups. For instance, the revolutionary work of Adams \cite{Ada66} on the Hurewicz homomorphism for real topological K-theory produced two $8$-periodic infinite families of elements. More recently, Behrens--Hill--Hopkins--Mahowald \cite{BHHM20} and Behrens--Mahowald--Quigley \cite{BMQ23}  used the Hurewicz homomorphism for connective topological modular forms, $\tmf$, to produce many new $192$-periodic infinite families of elements.

Together with Bhattacharya \cite{BBQ24}, we recently took these methods further and produced seven new infinite families in the stable homotopy groups of spheres. These elements were constructed using $\tmf$ and the mod two Moore spectrum, $M$, i.e., the suspension spectrum of the real projective plane. By virtue of their construction, these elements were all simple $2$-torsion. 

The starting point of the analysis in \cite{BBQ24} was the fact that the Hurewicz homomorphism $\pi_*A_1 \to \tmf_*A_1$ is surjective \cite{Pha23}. We then passed from $\pi_* A_1$ to $\pi_*S$ through the long exact sequences in homotopy and $\tmf$-homology associated to the cofiber sequences building $A_1$ from $S$:
\begin{align}
&S\xra{2} S \to V(0) ,\label{eq:S-S-V0} \\
&\Sigma V(0) \xra{\eta} V(0) \to Y ,\label{eq:V0-V0-Y} \\
&\Sigma^2 Y \xra{v_1} Y \to A_1. \label{Y-Y-A1}
\end{align}

In this work, we make a simple modification, using the long exact sequences associated to a different set of cofiber sequences building $A_1$ from $S$:
\begin{align}
&\Sigma S\xra{\eta} S \xra{i_1} C,  \label{eq:S-S-C}\\
&C \xra{2} C \xra{i_2} Y,  \label{eq:C-C-Y} \\ 
&\Sigma^2 Y \xra{v_1} Y \xra{i_3} A_1. \label{eq:Y-Y-A1}
\end{align}
Applying the techniques of \cite{BBQ24}, we find five infinite families with trivial $\tmf$-Hurewicz image.

\begin{thmx}\label{MT}
For each $m \in \{23, 73, 95, 120, 145\}$ and  $k \in \n$, there exists a simple $\eta$-torsion element in dimension $m + 192k$ of the stable stems whose image is trivial under the $\tmf$-Hurewicz homomorphism. 
\end{thmx}

\begin{remark}
In \cite{BBQ24}, no elements were constructed in dimensions congruent to $73$, $120$, and $145$ modulo $192$, so these are new infinite families. On the other hand, it is likely that the elements in dimensions congruent to $23$ and $120$ modulo $192$ coincide with the elements in the same dimensions constructed in \emph{loc. cit.}
\end{remark}

\begin{remark}\label{Rmk:Adams}
When $k=0$, we expect the elements in dimensions $23$ and $73$ to be detected by $gh_1^3$ and $\Delta h_1 d_0 e_0^2$, respectively, in the Adams spectral sequence (cf. \cite{IWX23}). A rigorous proof should be possible using the geometric boundary theorem \cite[Lem. A.4.1]{Beh12} and the analysis of the Adams spectral sequences for $\tmf_*S$ and $\tmf_*C$ in \cite{BR21}. 
\end{remark}

\begin{remark}
Adapting the proof of \cite[Thm. B]{BBQ24}, it is easy to show that the elements from \cref{MT} have nontrivial images in the $T(2)$- and $K(2)$-local stable stems. 
\end{remark}

\begin{remark}
After the initial appearance of this paper, Carrick and Davies \cite{CD25} used a stronger detection spectrum
$$J_0(3) := \fib(\psi^3-1: \tmf \to \tmf)$$
to detect many new infinite families in the stable stems. It is likely that the infinite families in this work are a subset of the ones constructed in their work, though we will not attempt a precise comparison here. It is worth noting that the spectrum $J_0(3)$ may only be used to construct infinite families one degree away from families in $\tmf_*$, while the families detected using $C$ as in this work will be two degrees away from families in $\tmf_*$. A more detailed analysis of the $\tmf$-Hurewicz homomorphism for $C$ could yield infinite families in a different range of degrees. 
\end{remark}

Aside from increasing our understanding of the stable stems, \Cref{MT} has some interesting geometric consequences. An \emph{exotic sphere} is a smooth manifold which is homeomorphic, but not diffeomorphic, to the sphere with its standard smooth structure. An exotic sphere is called \emph{very exotic} if it does not bound a parallelizable manifold. By work of Kervaire and Milnor \cite{KM63}, almost every nontrivial element in the cokernel of the $J$-homomorphism $J: \pi_nO \to \pi_n^s$ detects a very exotic sphere. Motivated by a conjecture of Wang and Xu about (not necessarily very) exotic spheres, we conjectured in \cite{BBQ24} that very exotic spheres exist in all dimensions greater than $4$, except dimensions $5$, $6$, $11$, $12$, $27$, $43$, $56$, and $61$. 

The results of \cite{Ada66, BHHM20, BMQ23, BBQ24} imply that very exotic spheres exist in at least $99$ of the congruence classes of dimensions modulo $192$. These did not include dimensions congruent to $120$ modulo $192$, so \Cref{MT} implies the following.

\begin{corollary}
Very exotic spheres exist in at least $100$ of the congruence classes of dimensions modulo $192$.
\end{corollary}

There is another connection to geometric topology motivating this work. In \cite{Hsi66, Bru69, Bru71}, Brumfiel and W.-C. Hsiang related the stable cohomotopy groups of complex projective spaces to smooth free circle actions on exotic spheres. In \cite{BQ26}, T. Bauer and the second author use the $\tmf$-Hurewicz homomorphism to study these cohomotopy groups, detecting infinite families of exotic spheres with smooth free circle actions. A typical question in the analysis is if lifts of certain elements in the $\tmf$-Hurewicz image are simple $\eta$-torsion in the stable stems. In technical terms, this condition ensures that an associated class in the stable cohomotopy of $\c P^{2n-1}$ extends to a class in the stable cohomotopy of $\c P^{2n}$. 

In other words, knowledge of whether or not an element with nontrivial $\tmf$-Hurewicz image is simple $\eta$-torsion a key ingredient in understanding smooth free circle actions on exotic spheres. This information can be gleaned from the long exact sequence associated to the cofiber sequence \eqref{eq:S-S-C}, which we analyze heavily in this work (building on previous results of Bruner and Rognes \cite{BR21}). Our results are summarized in the following. 

\begin{thmx}\label{MTb}
There exists a simple $\eta$-torsion element in the stable stems with $\tmf$-Hurewicz image $\Delta^{8k}x$, $k \geq 0$, where
$$x \in \{\kappa \nu, 4 \bar{\kappa}, \bar{\kappa}^2 \eta^2, \kappa \cdot \Delta^2 \nu, \bar{\kappa} \cdot 4\Delta^2, \bar{\kappa}^4, \bar{\kappa}^3 \cdot \eta \Delta, \bar{\kappa}^2 \cdot \eta^2\Delta^2, \bar{\kappa}^5, \bar{\kappa}^4 \cdot \eta \Delta, $$
$$\quad \quad \bar{\kappa}^3 \cdot \eta^2 \Delta^2, \nu^2 \cdot \kappa \Delta^4, \bar{\kappa}^5 \cdot \eta \Delta, \bar{\kappa}^5 \cdot \eta^2 \Delta^2, \bar{\kappa} \cdot 4 \Delta^6\}.$$
\end{thmx}

\begin{remark}
Many of the above elements can be written as products of elements in the $\tmf$-Hurewicz image where one factor is simple $\eta$-torsion. For these elements, one can prove $\eta \cdot \Delta^{8k}x = 0$ more easily using relations in the low-dimensional stable stems like $\nu \eta = 0$, $2\eta = 0$, and $\eta \cdot \bar{\kappa}^3 = 0$. For instance, we have $\eta \cdot \Delta^{8k}\kappa \nu = \eta \cdot \nu \cdot \Delta^{8k}\kappa = 0$ and  $\eta \cdot 4 \bar{\kappa} = \eta \cdot 2 \cdot 2 \bar{\kappa} = 0$. However, we cannot decompose the elements
$$\bar{\kappa}^2 \eta^2, \  \kappa \cdot \Delta^2 \nu,  \ \bar{\kappa}^2 \cdot \eta^2 \Delta^2$$
as products of elements in the $\tmf$-Hurewicz image where one factor lifts to a simple $\eta$-torsion element. Our more involved analysis using $A_1$ is the only proof we know that $\Delta^{8k}x$ admits a simple $\eta$-torsion lift in the stable stems for these three choices of $x$. 
\end{remark}

\subsection{Similarities and differences with previous work}

The general strategy and first step in this work closely mirrors that of our previous work \cite{BBQ24}. Starting with Pham's result that $\pi_*A_1 \to \tmf_*A_1$ is surjective, we use the long exact sequences in homotopy and $\tmf$-homology associated to the cofiber sequences \eqref{eq:S-S-C}, \eqref{eq:C-C-Y}, and \eqref{eq:Y-Y-A1}, similar to the use of the cofiber sequences \eqref{eq:S-S-V0}, \eqref{eq:V0-V0-Y}, and \eqref{Y-Y-A1} in \cite{BBQ24}, to pass from $\pi_*A_1$ to $\pi_*S$. Since the cofiber sequences \eqref{eq:Y-Y-A1} and \eqref{Y-Y-A1} are the same, the first step of passing from $A_1$ to $Y$ is already taken care of by \cite{BBQ24}; indeed, the left-most column of the table summarizing our new calculations (\cref{Table:A1lifts}) is the same as that of \cite[Table 2]{BBQ24}. 

The passage from $Y$ to $S$ is where this paper diverges significantly from our previous work, both in outline and in execution. As described above, in this paper we use \eqref{eq:S-S-C} and \eqref{eq:C-C-Y} to pass from $Y$ to $S$ instead of using \eqref{eq:S-S-V0} and \eqref{eq:V0-V0-Y}. This seemingly minor change has a major effect on how we analyze the resulting long exact sequences. 

In \cite{BBQ24}, we benefited greatly from existing knowledge \cite{Bau08,BBPX22} of the Adams--Novikov spectral sequences converging to $\tmf_*S$, $\tmf_*V(0)$, and $\tmf_*Y$. This compatibility between the ways these groups were computed gave us access to simple but powerful filtration arguments \cite[Technique 3]{BBQ24} which we used to determine the fates of classes in $\tmf_*Y$ in the long exact sequence associated to \eqref{eq:V0-V0-Y}. 

In this paper, we start with classes in $\tmf_*Y$, which has only been computed using the Adams--Novikov spectral sequence \cite{BBPX22}, but pass through $\tmf_*C$, which has only been computed using the Adams spectral sequence \cite{BR21}. While these spectral sequence could be compared with quite sophisticated methods, it turns out that we are able to completely analyze the fate of the relevant classes from $\tmf_*Y$ in the long exact sequence associated to \eqref{eq:C-C-Y} using simpler methods. In particular,  linearity with respect to the action of $\nu \in \pi_3S$ and $\bar{\kappa} \in \pi_{20}S$ allows us to reduce to lower dimensions which can be handled using sparsity arguments.

\subsection*{Acknowledgments}

This work extends the ideas from \cite{BBQ24} which were developed with Prasit Bhattacharya. The authors thank their coauthor for many valuable discussions. The authors also thank Dan Isaksen for a discussion leading to \cref{Rmk:Adams}. This work was supported by NSF grants DMS-2135884, DMS-2239362, DMS-2414922, and DMS-2441241. 

\section{Methods}\label{Sec:Eff}

For the rest of the paper, everything is implicitly $2$-local. Let $S$ denote the sphere spectrum. The finite complexes $C$, $Y$, and $A_1$ are defined via cofiber sequences \eqref{eq:S-S-C}, \eqref{eq:C-C-Y} and \eqref{eq:Y-Y-A1}
where $\eta$ is the first Hopf invariant one map, $\cdot 2$ is the degree two map, and $v$ is a  $v_1$-self-map of $Y$ \cite{DM81}. When the choice of $v_1$-self-map affects the action of $v_1$ on an element in $\tmf_*Y$, we always pick the self-map which makes the element simple $v_1$-torsion. 

The cofiber sequences \eqref{eq:S-S-C}, \eqref{eq:C-C-Y} and \eqref{eq:Y-Y-A1} give rise to long exact sequences in $\tmf$-homology
\begin{equation}\label{eqn:SSC}
\cdots \to \tmf_{k-1}S \xrightarrow{\eta} \tmf_kS \xra{i_1} \tmf_k C \xra{p_1} \tmf_{k-2}S \to \cdots,
\end{equation}
\begin{equation}\label{eqn:CCY}
\cdots \to \tmf_{k}C \xra{2} \tmf_kC \xra{i_2} \tmf_kY \xra{p_2} \tmf_{k-1}C \to \cdots,
\end{equation}
\begin{equation}\label{eqn:YYA1}
\cdots \to \tmf_{k-2}Y \xra{v} \tmf_kY \xra{i_3} \tmf_kA_1 \xra{p_3} \tmf_{k-3}Y \to \cdots
\end{equation}
and analogous  long exact sequences in homotopy. These long exact sequences are related by the $\tmf$-Hurewicz homomorphism $\h: \pi_*X \to \tmf_*X$. If $x \neq 0 \in \tmf_*X$, we write $\widetilde{x}$ for any class in $\pi_*X$ such that $\h(\widetilde{x}) = x$. 

To obtain infinite families, we use the following analysis.

\begin{thm}\label{thm:families}
Let $a \in \tmf_k A_1$ be an element such that the class $c := p_2(p_3(a)) \neq 0 \in \tmf_{k-4}C$. 

\begin{enumerate}

\item Suppose $p_1(c) = 0 \in \tmf_{k-6}S$, so $c$ admits a lift along $i_1$ in \eqref{eqn:SSC}. If moreover $i_1^{-1}(c) \cap \operatorname{im}(\h) = \emptyset$, then there exists a $192$-periodic infinite family of nonzero elements in the stable stems
$$\{ \overline{s}_{k-6+192i} \in \pi_{k-6+192i}(S) : i \in \mathbb{N} \}$$
such that $\overline{s}_{k-6+192i} \neq 0$ and $\h(\overline{s}_{k-6+192i})=0$ for all $i \in \mathbb{N}$. 

\item Suppose $p_1(c) \neq 0 \in \tmf_{k-6}S$ is in the image of the Hurewicz homomorphism $\h$. Then there exists a $192$-periodic infinite family of simple $\eta$-torsion elements in the stable stems
$$\{ \tilde{s}_{k-6+192i} \in \pi_{k-6+192i}(S) : i \in \mathbb{N} \}$$
such that $\tilde{s}_{k-6+192i} \neq 0$ and $\h(\tilde{s}_{k+6+192i}) = \Delta^i p_1(c)$ for all $i \in \n$.

\end{enumerate}

\end{thm}

\begin{proof}
We give a detailed proof of (1); the proof of (2) is similar. From the long exact sequence \eqref{eq:Y-Y-A1} we see that the classes $p_3(a)$ are precisely the elements in $\pi_*Y$ which are simple $v_1$-torsion. Since any $a \in \tmf_*A_1$ is in the image of the Hurewicz map, we have that any element of the form $p_3(a) \in \tmf_*Y$ (for some $a \in \tmf_*A_1$) is in the Hurewicz image as well.  
Then the commutative diagram 
	\begin{equation} \label{LES1}
	\begin{tikzcd}
		\cdots \arrow{r} & \pi_{k-3} C \arrow{r}{i_2} \dar["{\sf h}"'] & \pi_{k-3} Y \arrow{r}{p_2} \arrow[d, "{\sf h}"'] & \pi_{k-4} C \rar["\cdot 2"] \arrow{d}{\sf h} & \cdots \\
		\cdots \arrow{r} & \tmf_{k-3} C \rar["i_2"'] &\tmf_{k-3} Y \rar["p_2"'] & \tmf_{k-4} C\rar["\cdot 2"'] & \cdots
	\end{tikzcd}
\end{equation}
shows that $c=p_2(p_3(a))$ is also in the image of $\h$. Then we turn to the commutative diagram 
\begin{equation} \label{LES2}
	\begin{tikzcd}
		\cdots \arrow{r} & \pi_{k-4} S \arrow{r}{i_1} \dar["{\sf h}"'] & \pi_{k-4} C \arrow{r}{p_1} \arrow[d, "{\sf h}"'] & \pi_{k-6} S \rar["\cdot \eta"] \arrow{d}{\sf h} & \cdots \\
		\cdots \arrow{r} & \tmf_{k-4}S\rar["i_1"'] &\tmf_{k-4} C \rar["p_1"'] & \tmf_{k-6}S \rar["\cdot \eta"'] & \cdots
	\end{tikzcd}
\end{equation}
If $i_1^{-1}(c)$ is not in the image of $\h$, then the commutativity of \eqref{LES2} implies that there exists a simple $\eta$-torsion element $\overline{s}_{k-6} \in \pi_{k-6}S$ such that $\h(\overline{s}_{k-6})=0$. The proof is completed by using  192-periodicity of $\tmf$. 
\end{proof}

\subsection{$v_1$-periodic families}\label{Notn:v1}

We may ignore most $v_1$-periodic elements throughout our computations. We write
$$\tmf_*(X)^{\mathrm{tor}} := \ker\left( \ell: \tmf_*(X) \to v_1^{-1} \tmf_*(X) \right)$$
for the subgroup of $v_1$-torsion elements. Since every element of $tmf_*A_1$ is $v_1$-torsion and  $v_1$-torsion elements cannot map to $v_1$-periodic elements, we have the following lemma.

\begin{lemma}
For any nonzero element $a \in \tmf_*A_1$, we have
\begin{enumerate}
\item $p_3(a) \in \tmf_*(Y)^{\mathrm{tor}}$,
\item $p_2(p_3(a)) \in \tmf_*(C)^{\mathrm{tor}}$,
\item $p_1(p_2(p_3(a))) \in \tmf_*(S)^{\mathrm{tor}}$.
\end{enumerate}
\end{lemma}

\begin{notation}
We will refer to an element in $\tmf_*Y$ in stem $s$ which is detected in Adams--Novikov filtration $f$ by $y_{s,f}$. In the bidegrees we are interested in, there is only one nonzero $v_1$-torsion element, so these names uniquely determine elements up to higher Adams--Novikov filtration. 

Since Adams--Novikov spectral sequence computations of $\tmf_*C$ are not publicly available, we will refer to elements in $\tmf_*C$ using the names from \cite[Sec. 12.2]{BR21}. When $\tmf_s(C)^{\mathrm{tor}} \cong \mathbb{F}_2$, we will refer to the generator by $c_s$. 
\end{notation}

\subsection{Techniques}

In view of \cref{thm:families}, our goal is to find elements $a \in \tmf_k A_1$ such that $c := p_2(p_3(a)) \neq 0 \in \tmf_{k-4}C$, so we must analyze the long exact sequences \eqref{eqn:SSC}, \eqref{eqn:CCY}, and \eqref{eqn:YYA1}.

Analyzing \eqref{eqn:YYA1} is straightforward. An element $y \in \tmf_{k-3}Y$ is in the image of $p_3$ for some version of $A_1$ if and only if $v_1 \cdot y = 0 \in \tmf_{k-1}Y$ for a choice of $v_1$. The action of all choices of $v_1$ was computed in Figures 21-22 of  \cite{BBPX22}, so the image of $p_3$ is easily determined. These elements are listed in the left-most column of \cref{Table:A1lifts}.

To analyze \eqref{eqn:CCY}, we will use three main techniques. 

\begin{technique}[Vanishing kernel]\label{T1}
If $\ker(\cdot 2: \tmf_{k-1}(C)^{\mathrm{tor}} \to \tmf_{k-1}(C)^{\mathrm{tor}}) = 0$, then $p_2(y) = 0$ for any $y \in \tmf_kY$. If moreover $\tmf_k(Y)^{\mathrm{tor}} = \f_2\{y\}$ and $\coker(\cdot 2: \tmf_k(C) \to \tmf_k(C)) \cong \f_2\{c\}$, then $i_2(c) = y$. 
\end{technique}

\begin{technique}[Vanishing cokernel]\label{T2}
If $y \neq 0 \in \tmf_k(Y)^{\mathrm{tor}}$ and $\coker(\cdot 2: \tmf_k(C) \to \tmf_k(C)) = 0$, then $p_2(y) \neq 0$. Moreover, if $\ker(\cdot 2: \tmf_{k-1}(C)^{\mathrm{tor}} \to \tmf_{k-1}(C)^{\mathrm{tor}}) = \f_2\{c\}$, then $p_2(y) = c$. 
\end{technique}

For the next technique, note that $\tmf_*C$ and $\tmf_*Y$ are modules over $\tmf_*S$ and the maps in \eqref{eq:S-S-C} and \eqref{eq:C-C-Y} are $\tmf_*S$-linear. 

\begin{technique}[$\tmf_*S$-linearity]\label{T3}
Let $x \in \{\eta, \nu, \bar\kappa\} \subseteq \tmf_*S$. Suppose $y = x \cdot z \in \tmf_kY$ for some $z \in \tmf_*Y$. Then $p_2(y) = p_2(xz) = x \cdot p_2(z)$. In \Cref{Table:A1lifts} we will sometimes write $(3x)$ in the third column to clarify which $x$ has been used. 
\end{technique}

In \cref{Table:A1lifts}, we list the image of $p_2$ in the second column and the relevant technique in the third column. Applying Techniques 1-3 is fairly straightforward, so we will not write out the details here. However, there are a few cases that require further analysis. We analyze them in the following lemmas and then list the relevant lemma in the table. 
\begin{lemma}\label[lemma]{lem:15}
	We have 
	\begin{enumerate}
		\item $p_2(y_{15,1})=i(d_0)$
		\item  $p_2(y_{35, 5})=gi(d_0)$
		\item  $p_2(y_{55,9})=g^2i(d_0)$
		\item $p_2(y_{18, 2})= h_0i(e_0)$
		\item $i_2(g^2(i(\beta g)))= y_{75,13}$
		\item $p_2(y_{55,7})=g^2i(d_0)$
		\item  $p_2(y_{35, 3})=gi(d_0)$
	\end{enumerate}
\end{lemma}

\begin{proof}
	The first statement is obvious from analysis of the long exact sequence \eqref{eq:C-C-Y}. The second statement follows from the first statement and the facts  that $y_{35, 5}=\bar{\kappa} y_{15,1}$  and that $g$ detects $\bar{\kappa}$ in the Adams spectral sequence. 
	Note also that both $y_{35, 5}$ and $gi(d_0)$ are the classes in the top filtration in the corresponding stems, so there is no ambiguity with classes being defined up to higher filtration. 
	
	Statement (3) follows completely analogously to (2). For (4) we use (1) together with $\nu$-linearity. 
	
	As a consequence of (2) we have that either $y_{35, 3}$ or $y_{35,3}+y_{35,5}$ maps to $0$ under $p_2$. Similarly, as a consequence of (3) we have that either $y_{55, 7}$ or $y_{55,7}+y_{55,9}$ maps to $0$ under $p_2$. In order to determine which of these happen, we multiply by $\bar{\kappa}$ once again and consider stem 75. Here there is only one possibility: due to sparseness we have  $i_2(g^2 i(\beta g))=y_{75,13}=\bar{\kappa}y_{55,9}$.  This implies that in stem 55 we have $i_2(g i(\beta g))=y_{55,7}+y_{55,9}$ and (6) follows from exactness. Then (7) follows by $\bar{\kappa}$-linearity. 
  \end{proof}

\begin{lemma}\label[lemma]{lem:56}
	The following claims hold: 
	\begin{enumerate}
		\item $p_2(y_{56,6})=g \cdot i(\beta g)$
		\item $p_2(y_{76,10})=g^2 \cdot i(\beta g)$ 
	\end{enumerate}
\end{lemma}

\begin{proof}
 For the first claim note that $\bar{\kappa} \cdot y_{56,6} \neq 0$ and $\bar{\kappa} \cdot y_{56,0}=0$ for any class $y_{56,0}$.  Since $\bar{\kappa}$-torsion classes cannot map to $g\cdot i(\beta g)$ (which supports a non-trivial product with $\bar{\kappa}$), the only option is $p_2(y_{56,6})=g \cdot i(\beta g)$.
 The second claim follows by $\bar{\kappa}$-linearity. 
 \end{proof}
\begin{lemma}\label[lemma]{lem:86-98}
	The following claims hold:
	\begin{enumerate}
		\item $p_2(y_{46,4})=g^2\widehat{h}_2$
		\item $p_2(y_{66,8})=g^3\widehat{h}_2$
		\item $p_2(y_{86,12})=g^4\widehat{h}_2$
			\item $p_2(y_{66,2})=c_{65}:=\beta^2g\widehat{\beta}+d_0w_2\widehat{h}_2$
		\item $p_2(y_{98,4})=c_{97}$
	\end{enumerate}
\end{lemma}
\begin{proof}
	As in the previous lemma, we note that $y_{46,4}$ is the only class in $\tmf_{46}Y$ which is not simple $\bar{\kappa}$-torsion. The class $g^2\widehat{h}_2\in \tmf_{45}C$ is in the kernel of multiplication by $2$, so it is in the image of $p_2$, and, since it is not a simple $\bar{\kappa}$-torsion class, (1) follows. 
 	
 	Together with $\bar{\kappa}$-linearity, (1) implies (2) and (3). 
 
	The class $c_{65} \in \tmf_{65}C$ is not simple $\nu$-torsion. The only class in $\tmf_{66}Y$ which is not simple $\nu$-torsion is   $y_{66,2}$, which implies (4).

	 To prove (5), first, we divide by $\bar{\kappa}$ and consider $y_{78,0}$. By the same arguments  as above, since $y_{78,0}$ is the only class in stem 78 which is not simple $\bar{\kappa}$-torsion, in the long exact sequence \eqref{eq:C-C-Y} it maps into the unique non-simple $\bar{\kappa}$ torsion class in $\tmf_{77}C$, which we denote by $c_{77}$. Then $\bar{\kappa}$ linearity implies statement (2).
\end{proof}

\begin{lemma}\label[lemma]{lem:JD}
The following claims hold:
\begin{enumerate}
\item $p_2(y_{20,2}) = d_0 \widehat{h}_2$;
\item $p_2(y_{45,3}) = g i(\alpha^2)$ and $p_2(y_{45,9}) = 0$;
\item $p_2(y_{51,1}) = 0$;
\item $p_2(y_{56,2}) = 0$;
\item $p_2(y_{60,12}) = 0$;
\item $p_2(y_{68,2}) = c_{67}$;
\item $p_2(y_{71,9}) = gi(\beta^2)$.
\end{enumerate}
\end{lemma}

\begin{proof}
We will use the process of elimination: 
\begin{enumerate}
\item Suppose $\tmf_k(Y)^{\mathrm{tor}} = \f_2\{y_1,y_2\}$, $\ker( \tmf_{k-1}(C)^{\mathrm{tor}} \xrightarrow{\cdot 2} \tmf_{k-1}(C)^{\mathrm{tor}}) = \f_2\{c_{k-1}\}$, and $\coker(\tmf_k(C) \xrightarrow{\cdot 2} \tmf_k(C)) \cong \f_2\{c_k\}$. If $i_2(c_k) = y_1$, then $p_2(y_2) = c_{k-1}$. 
\item 
Similarly, if $\tmf_k(Y)^{\mathrm{tor}} = \f_2\{y_1,y_2\}$, $\ker( \tmf_{k-1}(C)^{\mathrm{tor}} \xrightarrow{\cdot 2} \tmf_{k-1}(C)^{\mathrm{tor}}) = \f_2\{c_1,c_2\}$, and $\coker(\tmf_k(C) \xrightarrow{\cdot 2} \tmf_k(C)) = 0$, then if $p_2(y_1) = c_1$, then $p_2(y_2) = c_2$ or $p_2(y_2) = c_2 + c_1$. In the latter case, we may change the basis and replace $y_2$ by $y_1 + y_2$ so that the image is just~ $c_2$. 
\end{enumerate}
We apply these arguments for the case in the lemma as follows: 

\begin{enumerate}

\item We have $i_2(\bar{\kappa} \cdot 1) = \bar{\kappa} \cdot i(1) = y_{20,4}$ by \cref{T3}. So $p_2(y_{20,4})=0$ and $p_2(y_{20,2}) \neq  0$; the only possibility is $p_2(y_{20,2}) = d_0 \widehat{h}_2$. 

\item By \cref{T1}, $i_2(g \cdot \widehat{h}_2) = y_{25,5}$, so by \cref{T3}, $i_2(\bar{\kappa} \cdot g \cdot \widehat{h}_2) = \bar{\kappa} y_{25,5} = y_{45,9}$. Therefore $p_2(y_{45,3}) = g \cdot i(\alpha^2)$, and it follows immediately that $p_2(y_{45,9}) = 0$. 

\item We have $p_2(y_{51,1})=0$ because the only possible target is $g \cdot i(\beta^2)$, but  $y_{51,1}$ is simple $\bar{\kappa}$-torsion and  $g \cdot i(\beta^2)$ is not. 

\item By \cref{T2}, $p_2(y_{36,2}) = i(\beta g)$, so by \Cref{T3}, $p_2(y_{56,6}) = g \cdot i(\beta g)$. 

\item \cref{T3} implies $p_2(y_{60,10}) = g^2 d_0 \widehat{h}_2$, so $p_2(y_{60,12}) = 0$.

\item By \cref{T1}, $i_2(i(h_0w_2)) = y_{48,4}$, so by \cref{T3}, $i_2(g \cdot i(h_0w_2)) = y_{68,8}$. Therefore $p_2(y_{68,2}) = \beta^2 g \widehat{\beta} + d_0 w_2 \widehat{h}_2 =: c_{67}$. 

\item \cref{T3} implies $p_2(y_{71,3}) = g \cdot h_0 w_2 \widehat{h}_0$, so $p_2(y_{71,9}) = g \cdot i(\beta^2)$. 

\end{enumerate}
\end{proof}

Analyzing the long exact sequence \eqref{eqn:SSC} is easy using Bruner and Rognes' notation  in Figures 12.9-12.16 of \cite{BR21}: an element containing $\widehat{c}$ will have nontrivial image under $p_1$, while an element containing $i(s)$ will have preimage $s$. After converting Adams spectral sequence names to the Adams--Novikov names from \cite{DFHH14}, we list the images under $p_1$ and preimages under $i_1$ of the elements from $\operatorname{img}(p_2)$ in columns 4 and 5, respectively, of \cref{Table:A1lifts}.

\section{Summary table}

We summarize our calculations in \Cref{Table:A1lifts} as follows. The leftmost column lists the image of $p_3$ in $\tmf_*Y$, i.e. it consists of simple $v_1$-torsion elements, which, for example, can be read off of Figures 8 and 9 in \cite{BBPX22}. We determine their image in column $2$ and indicate the technique used, among \Cref{T1} through \Cref{T3}, in column $3$.  

We calculate the image  under $p_1$ of nonzero elements in column $2$ and record them in column $4$. If the image is zero, we identify an $\eta$-torsion element which is its lift along $i_1$ and record it in column $5$. These calculations are immediate from \cite{BR21}, so we do not list techniques.

\begin{longtable}{| l| | l l| | l |l|}
\caption{Detecting elements in $\tmf_*$}  \label{Table:A1lifts} \\
\toprule
$\mr{img}(p_3)$   & $\mr{img}(p_2)$ & (T)  & $\mr{img}(p_1)$   &  $i_1^{-1}(-)$  \\
\midrule 
${\sf y}_{3,1}$ & $0$ & (1)  &  &  \\
${\sf y}_{6,2}$ & $0$ &  (3) &  &  \\
${\sf y}_{8,2}$ & $0$ &  (1) &  &    \\
${\sf y}_{11,3}$ & $0$ &  (1) & & \\
${\sf y}_{14,2}$ & $0$ & (1) && \\
${\sf y}_{18,2}$ & $h_0 \cdot i(e_0)$ &  \Cref{lem:15} & $0$ & $\kappa \nu$ \\
${\sf y}_{20,2}$ & $d_0 \widehat{h}_2$ &  \Cref{lem:JD} & $\kappa \nu$ & \\
${\sf y}_{21,3}$ & $h_0^2 g i(1)$ &  $(3\eta)$ & $0$ & $4 \bar{\kappa}$ \\
${\sf y}_{23,3}$ & $h_0 g \widehat{h}_0$ &   $(3\nu)$ & $4 \bar{\kappa}$  &   \\
${\sf y}_{26,4}$ & $g \widehat{h}_2$ &  $(3\nu)$& $0$ & $\eta \Delta$ \\
${\sf y}_{29,5}$ & $0$ &  (1) && \\
${\sf y}_{34,6}$ & $0$ &  (1)   && \\
${\sf y}_{35,3}$ & $gi(d_0)$  &   \Cref{lem:15}   &0 &$\bar{\kappa}\kappa$ \\
${\sf y}_{39,7}$ & $0$ &   (1)   &&  \\
${\sf y}_{40,6}$ & $\bar{\kappa} \cdot d_0 \widehat{h}_2$ & $(3\bar{\kappa})$  & $0$ & $\kappa \cdot \eta \Delta$ \\
${\sf y}_{44,8}$ & $0$ &   (1)  & &  \\
${\sf y}_{45,3}$ & $g i(\alpha^2)$ &  \Cref{lem:JD}  & $\bar{\kappa}^2 \eta^2$ & \\
${\sf y}_{45,9}$ & $0$ & \Cref{lem:JD} &&\\
${\sf y}_{50,4}$ & $0$ & (1) & & \\
${\sf y}_{50,6}$ & $0$ & (1) & &     \\
${\sf y}_{51,1}$ & $0$ & \Cref{lem:JD}  & &  \\
${\sf y}_{54,2}$ & $0$ & $(3\nu)$     & & \\
${\sf y}_{55,7}$ & $g^2 i(d_0)$  & \Cref{lem:15} & $0$ & $\bar{\kappa}^2 \kappa$ \\
${\sf y}_{56,2}$ & $0$ & \Cref{lem:JD}  & &  \\
${\sf y}_{57,11}$ & $0$ & $(3\nu)$ & &   \\ 
${\sf y}_{59,3}$ & $0$ & $(3\nu)$ & & \\ 
${\sf y}_{60,10} $ & $\bar{\kappa}^2 \cdot d_0 \widehat{h}_2$ & $(3\kappabar)$ & $0$ & $\bar{\kappa} \kappa \cdot \eta \Delta$  \\
${\sf y}_{60,12} $ & $0$ & \Cref{lem:JD} & &   \\
${\sf y}_{62,2}$ & $0$ & (1) & &  \\
${\sf y}_{65,7}$ & $0$ &(1) &&\\
${\sf y}_{65,13}$ & $0$ & $(3\bar{\kappa})$ &&\\
${\sf y}_{66,2}$ & $c_{65}$ &\Cref{lem:86-98} & $0$ & $\kappa \cdot \nu \Delta^2$ \\ 
${\sf y}_{68,2}$ & $c_{67}$  &\Cref{lem:JD} & $\kappa \cdot \nu \Delta^2$ &   \\
${\sf y}_{69,3}$ & $h_0 g i(h_0w_2)$ &(2) & $0$ & $\bar{\kappa} \cdot 4 \Delta^2$  \\
${\sf y}_{70,8}$ & $0$ & (1)  & &  \\
${\sf y}_{70,10}$ & $ 0$ &(1)  & &\\
${\sf y}_{71,3}$ & $\nu \cdot c_{67}$ &$(3\nu)$ & $\bar{\kappa} \cdot 4 \Delta^2$ & \\
${\sf y}_{71,9}$ & $g i(\beta^2) $ & \Cref{lem:JD} & $0$ & $\bar{\kappa} \cdot \eta^2 \Delta^2$  \\
${\sf y}_{74,4}$ & $0$ & (1) & & \\ 
${\sf y}_{75,13}$ & $0$ &(1)&&\\
${\sf y}_{76,10}$ & $g^2 i(\beta g)$ &\Cref{lem:56}& $0$ & $(\eta \Delta)^3$ \\ 
${\sf y}_{77,5}$ & $0$ & (1) & & \\
${\sf y}_{80,16}$ & $0$ &(1) && \\
${\sf y}_{81,11}$ & $g^4 i(1)$ &(2) & $0$ & $\bar{\kappa}^4$ \\
${\sf y}_{82,6}$ & $0$ & (1) & & \\
${\sf y}_{83,3}$ & $c_{82}$ & (2) & $\bar{\kappa}^4$ & \\
${\sf y}_{85,17}$ & $0$ &$(3\bar{\kappa})$ &&\\
${\sf y}_{86,12}$ & $g^4 \widehat{h}_2$ &\Cref{lem:86-98} & $0$ & $\bar{\kappa}^3 \cdot \eta \Delta$  \\
${\sf y}_{87,7}$ & $0$ & (1) & & \\
${\sf y}_{88,6}$ & $\bar{\kappa} \cdot c_{67}$ &(3)& $\bar{\kappa}^3 \cdot \eta \Delta$ & \\
${\sf y}_{90,14}$ & $0$ &(1)&&\\
${\sf y}_{91,13}$ & $g^3 \cdot i(\beta^2)$ &(2) & $0$ & $\bar{\kappa}^2 \cdot \eta^2 \Delta^2$ \\
${\sf y}_{92,8}$ & $0$ & (1) & &  \\
${\sf y}_{93,3}$ & $c_{92}$ &(2) & $\bar{\kappa}^2 \cdot \eta^2 \Delta^2$ &  \\
${\sf y}_{96,14}$ & $0$ &(1) &&\\
${\sf y}_{97,9}$ & $0$ &(3) &&\\
${\sf y}_{98,4}$ & $c_{97}$ & \Cref{lem:86-98} & $0$ & $\eta \Delta^4$  \\
${\sf y}_{101,15}$ & $0$ &(1) &&  \\
${\sf y}_{102,2}$ & $0$ & (3) &&\\
${\sf y}_{102,10}$ & $0$ &(3)  & & \\
${\sf y}_{103,7}$ & $\bar{\kappa} \cdot c_{82}$ &(3) & $\bar{\kappa}^5$  &  \\
${\sf y}_{105,21}$ & $0$ &(3) &&\\
${\sf y}_{106,16}$ & $\bar{\kappa} \cdot g^4 \widehat{h}_2$ &(3)& $0$ & $\bar{\kappa}^4 \cdot \eta \Delta$ \\
${\sf y}_{107,3}$ & $0$ & (1) & &  \\
${\sf y}_{107,11}$ & $0$ & (1) & & \\
${\sf y}_{108, 10}$  & $\bar{\kappa}^2 \cdot c_{67}$ & (3) & $\bar{\kappa}^4 \cdot \eta \Delta$ & \\
${\sf y}_{111,17}$ & $c_{110}$ &(2)& $0$ & $\bar{\kappa}^3 \cdot \eta^2 \Delta^2$ \\
${\sf y}_{112,12}$ & $0$ & (1) & & \\
${\sf y}_{113,7}$ & $\bar{\kappa} \cdot c_{92}$ & (3) & $\bar{\kappa}^3 \cdot \eta^2 \Delta^2$ & \\
${\sf y}_{116,18}$ & $0$ &(3) &&\\
${\sf y}_{117,3}$ & $c_{116}$ & (2,3) & $0$ & $\nu^2 \cdot \kappa \Delta^4$ \\
${\sf y}_{117,13}$ & $0$ &(3) &&\\
${\sf y}_{118,8}$ & $\bar{\kappa} \cdot c_{97}$ & (3) & $0$ & $\bar{\kappa} \cdot \eta \Delta^4$ \\
${\sf y}_{119,3}$ & $c_{118}$ & (2) & $\nu^2 \cdot \kappa \Delta^4$ & \\
${\sf y}_{122,4}$ &$0$ & (1) & &  \\
${\sf y}_{122,14}$ & $0$ & (1)  & & \\
${\sf y}_{123,11}$ & $\bar{\kappa}^2 \cdot c_{82}$ & (3) & $0$ & $\eta^2 \Delta^5$ \\
${\sf y}_{127,15}$ & $0$ & (1) & &\\
${\sf y}_{128,14}$ & $\bar{\kappa}^3 \cdot c_{67}$ & (3) & $\bar{\kappa}^5 \cdot \eta \Delta$ & \\
${\sf y}_{132,16}$ & $0$ & (1) &&\\
${\sf y}_{133,11}$ & $\bar{\kappa}^2 \cdot c_{92}$ & (3) & $\bar{\kappa}^4 \cdot \eta^2 \Delta^2$ &  \\
${\sf y}_{137,17}$ & $0$ &(3) &&\\
${\sf y}_{138,12}$ & $\bar{\kappa}^2 \cdot c_{97}$ & (3) & $0$ & $\kappa \cdot 2\nu \Delta^5$ \\
${\sf y}_{142,18}$ & $0$ &(1) &&\\
${\sf y}_{143,15}$ & $\bar{\kappa}^3 \cdot c_{82}$ & (3) & $0$ & $\kappa \Delta^4 q$ \\
${\sf y}_{148,18}$ & $\bar{\kappa}^4 \cdot c_{67}$ & (3) & $0$ & $\eta \Delta^6$ \\
${\sf y}_{150,2}$ & $0$ &(1,3) &&\\
${\sf y}_{153,11}$ & $0$ & (3) &&\\
${\sf y}_{153,15}$ & $\bar{\kappa}^3 \cdot c_{92}$ & (3) & $\bar{\kappa}^5 \cdot \eta^2 \Delta^2$ & \\
${\sf y}_{155,3}$ & $0$ & (1) & & \\
${\sf y}_{158,16}$ & $0$ & (1) & & \\
${\sf y}_{161,7}$ & $0$ & (1) & & \\
${\sf y}_{165,3}$ & $c_{164}$ & (2) & $0$ & $\bar{\kappa} \cdot 4 \Delta^6$ \\
${\sf y}_{167,3}$ & $c_{166}$ & (2) & $\bar{\kappa} \cdot 4 \Delta^6$ & \\
${\sf y}_{168,22}$ & $0$ & (1) & & \\
${\sf y}_{170,4}$ & $0$ & (1) & & \\
\toprule
\end{longtable}
Given the data in this table and \Cref{thm:families} it is now easy to prove our main theorems.

\begin{proof}[{Proof of \Cref{MT}}]
	For any element $x \in$ \{$y_{26,4}$, $y_{76,10}$, $y_{98,4}$, $y_{123,11}$, $y_{148,18}$\} we have that the following are true: 
	\begin{enumerate}
		\item $p_2(x)\neq 0$
		\item $p_1(p_2(x))=0$
		\item $i_1^{-1}(p_2(x)) \in \pi_*\tmf$ is not in the Hurewicz image.
	\end{enumerate}
Then the result follows from \Cref{thm:families}(1).
\end{proof}

\begin{proof}[{Proof of \Cref{MTb}}]
	For any element $x$ in the given set the following can be read off of \Cref{Table:A1lifts}:
	\begin{enumerate}
		\item $x=p_1(p_2(p_3(a)))$ for some $a\in \tmf_*A_1$ 
		\item $x$ is in the image of the Hurewicz homomorphism
	\end{enumerate}
	The result follows from \Cref{thm:families}(2). 
\end{proof}
\bibliographystyle{alpha}
\bibliography{master}

\end{document}